\documentclass[12pt]{article}
\usepackage[usenames]{color} 
\usepackage{amssymb} 
\usepackage{amsmath} 
\usepackage{amssymb,bm,amsmath,amssymb,amsthm}
\usepackage[]{inputenc}

\usepackage{authblk}

\usepackage{hyperref}
\usepackage[small,nohug,heads=vee]{diagrams}
\diagramstyle[labelstyle=\scriptstyle]

\DeclareMathOperator{\Gal}{Gal}
\DeclareMathOperator{\Ind}{Ind}

\DeclareMathOperator{\Res}{Res}
\DeclareMathOperator{\Hom}{Hom}

\DeclareMathOperator{\Pgl}{PGL}
\DeclareMathOperator{\gl}{GL}
\DeclareMathOperator{\SL}{SL}
\DeclareMathOperator{\Psl}{PSL}

\DeclareMathOperator{\Cl}{Cl}

\newtheorem{cor}{Corollary}
\newtheorem{ex}{Example}

\newtheorem{lem}{Lemma}

\newtheorem{theorem}{Theorem}
\newcommand{\overbar}[1]{\mkern 1.5mu\overline{\mkern-1.5mu#1\mkern-1.5mu}\mkern 1.5mu}
\newcommand{\defeq}{\stackrel{\text{def}}{=}}

\begin{document}
\title{On Artin L-functions and Gassmann Equivalence for Global Function Fields}
\author{Pavel Solomatin \\
\texttt{p.solomatin@math.leidenuniv.nl}
}
\affil{Leiden University, Mathematical Department,\\
Niels Bohrweg 1, 2333 CA Leiden}

\date{ Leiden, 2016}
\maketitle

\begin{abstract}
In this paper we present an approach to study arithmetical properties of global function fields by working with Artin L-functions. In particular we recall and then extend a criteria of two function fields to be arithmetically equivalent in terms of Artin L-functions of representations associated to the common normal closure of those fields. We provide few examples of such non-isomorphic fields and also discuss an algorithm to construct many such examples by using torsion points on elliptic curves. Finally, we will show how to apply our results in order to distinguish two global fields by a finite list of Artin L-functions. 
  
\end{abstract}

\textbf{\\ \\ \\ \\ \\ \\ \\ Acknowledgements:} 
I would like to thank both my advisors, namely professor Bart de Smit and professor Karim Belabas, for their useful advices during the project.

\newpage

\section{Introduction}

Let $K$ and $L$ be two number fields. We will say that they \emph{split equivalently} if for any prime number $p \in \mathbb Z$ there exists a bijection $\phi_p$ from the set of primes in $\mathcal O_K$ lying above $p$ to the set of those primes in $\mathcal O_L$. We will say that they are \emph{arithmetically equivalently} if for each $p$ the bijection $\phi_p$ is degree preserving. Let $N$ denote the common Galois closure of $K$ and $L$ over $\mathbb Q$ and let $G= \Gal(N / \mathbb Q) $, $H = \Gal(N / K) $, $H' = \Gal(N / L) $. We will call a triple $(G,H, H')$ a \emph{Gassmann triple} if for any conjugacy class $[c]$ in $G$ we have $|[c] \cap H | = |[c] \cap H'|$ or equivalently if we have isomorphism of induced representation $ \Ind^{G}_{H} (1_{H}) = \Ind^{G}_{H'} (1_{H'})$, where $1_H$(and $1_{H'}$) means trivial representation of $H$(of $H'$ respectively). Then we have the following famous result, see $\cite{Perl}$ and $\cite{Perl2}$: 

\begin{theorem}[Perlis]
The following statements are equivalent: 
\begin{enumerate}
	\item $\zeta_K(s) = \zeta_{L}(s)$;
        \item  $K$ and $L$ are arithmetically equivalently;
        \item  $K$ and $L$ split equivalently; 
        \item 	$(G,H,H')$ form a Gassmann triple.  		
\end{enumerate}
\end{theorem}
   
We will call a Gassmann triple non-trivial if $H$ and $H'$ are not conjugate inside $G$. It happens if and only if $K$ is not isomorphic to $L$, as abstract fields or equivalently as extensions of $\mathbb Q$. This theorem gives rise to a variety of interesting results, for example by using group theory Perlis showed that if the degree of $K$ does not exceed 6, then equality $\zeta_K(s) = \zeta_L(s)$ implies $K \simeq L$. On the other hand there are infinitely many non-isomorphic pairs $(K_{\alpha} , L_{\alpha})$ of (isomorphism classes of) fields of degree seven such that $\zeta_{K_{\alpha}}(s) = \zeta_{L_{\alpha}}(s) $. Arithmetically equivalent fields share many common invariants, for example: normal closure, number of real and complex places, product of regulator and class number. But also, it is well known that class numbers itself may be different, see $\cite{Bart1}$. A good reference for the topic is \cite{Arith}. 
 
Now one can ask a similar question about global function fields. Let $q=p^m$, $p$ is prime and $k = \mathbb F_q$. Let us consider two curves $X$ and $Y$ over $k$. By a curve we mean smooth, projective, geometrically connected variety of dimension one over $k$. If we fix a $k$-rational separable map of $X$ and $Y$ to $\mathbb P^{1}$, then we obtain two finite separable geometric extensions of $\mathbb F_q(x)$, we will denote them by $K$ and $K'$ respectively. By analogy with the number field case, we have notions of arithmetical equivalence, splitting equivalence, Gassmann equivalence and Dedekind zeta-function. Then one checks that in this case notions of Gassmann equivalence, splitting equivalence and arithmetical equivalence coincide. 

But in contrast to the number field case, the above theorem is false in its full generality for the function field case. Namely, the implication from 1 to 2 is problematic. The problem is that the Dedekind zeta-function does not determinate spitting type, since in general there exist places in $K$ with the same norm above different places of $\mathbb F_q(x)$. One suitable approach here is to fix definition of the zeta-function associated to $K$. It turns out, that if one replace usual zeta-function by the so-called \emph{lifted Goss zeta-function}, then an analogue of the Perlis theorem becomes true. We refer an interested reader to \cite{Gunt1}.

The main purpose of this paper is to recall and then extend another approach to study arithmetically equivalent global function fields. Let  $K / F$ be a Galois extension of global fields with the Galois group $G = \Gal(K / F)$. Then for any finite dimensional complex representation $\rho$ of $G$ one attaches the Artin L-function $L_{F}(\rho, s)$. This is meromorphic function of complex variable $s$. For the sake of brevity we will denote them by $L_{F}(\rho)$. In 1986 K.Nagata published the paper \cite{Nagata} from which a careful reader could extract the following result:

\begin{theorem}\label{Nagata}
Let $K$, $K'$ denote two finite separable geometric extensions of $\mathbb F_q(x)$. Len $N$ denote the common Galois closure and $G= \Gal(N / \mathbb F_q(x)) $, $H = \Gal(N / K) $, $H' = \Gal(N / K') $. Let $\rho_1$, \dots $\rho_n$ denote all irreducible complex representations of $G$. Let $\psi = \Ind^{G}_{H}(1_H)$ and $\psi' = \Ind^{G}_{H'}(1_{H'})$. The following are equivalent: 
\begin{enumerate}
	\item For all  $i$ such that $1 \le i \le n$, we have $L_{K} ( \rho_i|_{H} ) = L_{K'} ( \rho_i|_{H'} )$; 
	\item $L_K(\psi|_{H})=L_{K'}(\psi|_{H'})$ and $L_K(\psi'|_{H})=L_{K'}(\psi'|_{H'})$;
        \item  $K$ and $L$ are arithmetically equivalently;	
        \item  $K$ and $L$ split equivalently; 
        \item 	$(G,H,H')$ form a Gassmann triple.  	
\end{enumerate}
\end{theorem}

In this paper we improve his argument and prove the above theorem as a particular case of the following more general result\footnote{ In order to get Nagata's result plug in the settings trivial representations $\alpha = 1_H$ and $\beta = 1_{H'}$}:

\begin{theorem}\label{main}\label{Lfunct}
In the above settings let $\alpha$ denotes a complex representation of $H$ and $\alpha'$ denotes a complex representation of $H'$. Let $\psi =\Ind^{G}_{H} (\alpha)$ and $\psi' =\Ind^{G}_{H'} (\alpha')$. Let $\bar{\rho}$ denotes the dual representation of $\rho$.The following are equivalent: 
\begin{enumerate}
	\item For all $i$ such that $1 \le i \le n$ we have equality of Artin L-functions: $L_{K} ( \alpha \otimes \rho_i|_{H} ) = L_{K'} ( \alpha' \otimes \rho_i|_{H'} )$ 
	\item $L_K({\bar{\alpha}}\otimes(\psi|_{H}) )=L_{K'}({\bar{\alpha'}}\otimes(\psi|_{H'}))$, and \newline
	$L_K({\bar{\alpha}}\otimes(\psi'|_{H} ))=L_{K'}({\bar{\alpha'}}\otimes(\psi'|_{H'}))$;
	\item Induced representations $\psi$ and $\psi'$  are isomorphic.
\end{enumerate}
\end{theorem}  

This theorem is not just a formal generalizations of Nagata's results but also allows us to use group theory to construct for any given pair of non-isomorphic global function fields a \emph{finite list of L-functions which distinguishes them}. This goal is achieved in two steps. First we need the following group-theoretical result. 
 
 Let $G$ be a finite group, $H$ a subgroup of index $n$, and $C_l=\mu_l$ be a cyclic group of order $l$, where $l$ is an odd prime. Let us consider semi-direct products $\tilde{G}=C_l^{n}\rtimes G$ and $\tilde{H}=C_l^{n}\rtimes H$, where $G$ acts on the components of $C_l^{n}$ by permuting them as cosets $G/H$. We will construct an abelian character $\chi$ of $\tilde{H}$ to the group $\mu_l$, such that the following is true: 
 
 \begin{theorem}[Bart de Smit]\label{Group}
For any  subgroup $\tilde{H}' \subset \tilde{G}$ and any character $\chi'$ : $\Tilde{H}' \to \mathbb C^{*}$ if $\Ind_{\tilde{H}'}^{\tilde{G}} (\chi') =  \Ind_{\tilde{H} }^{\tilde{G}} (\chi ) $ then $\tilde{H'}$ and $\tilde H$ are conjugate in $\tilde{G}$.
  \end{theorem}

\begin{bf}Remark:\end{bf} note that theorem [\ref{Group}] could be also applied to the number field case. More concretely for a given number field $K$, Bart de Smit constructs an abelian character $\chi$ of the Galois group of some Galois extension of $K$ such that Artin L-function $L_K(\chi)$ occurs only for that field. Finding an analogue of this result served as initial motivation for our article. The following theorem is one of our achievements in the direction of generalization of his results.

Next, in the settings of theorem [\ref{Nagata}] we construct a Galois extension $M$ of $\mathbb F_q(t)$ containing $K$ and $K'$ such that the Galois group $\Gal(M : \mathbb F_q(t))$ is $\tilde{G}$ and $K=M^{\tilde{H}}$, $K'=M^{\tilde{H'}}$ for $\tilde{G}$, $\tilde{H}$,  $\tilde{H'}$ as in theorem [\ref{Group}]. Altogether, this gives us:
 
\begin{theorem}\label{Mn}
For a given pair $K$ and $K'$ of finite separable geometric extensions of $F=\mathbb \mathbb F_q(t)$ there exists  a Galois extension $M$ of $\mathbb F_q(t)$ with Galois group $\tilde{G}$, such that $K = M^{\tilde{H}}$ and $K' = M^{\tilde{H'}}$ for some subgroups $\tilde{H}$, $\tilde{H'}$ of $\tilde{G}$ with the following properties. There exists an abelian character $\alpha$ of $\tilde{H}$ such that for any abelian character $\alpha'$ of $\tilde{H'}$ the following are equivalent :
\begin{enumerate}	
	\item For any irreducible representation $\rho$ of $\tilde{G}$ we have equality of Artin L-functions: $L_{K} ( \alpha \otimes \rho|_{\tilde{H}} ) = L_{K'} ( \alpha' \otimes \rho|_{\tilde{H'}} )$;
	\item $L_K({\bar{\alpha}}\otimes(\psi|_{\tilde{H}}) )=L_{K'}({\bar{\alpha'}}\otimes(\psi|_{\tilde{H'}}))$, and \newline
	$L_K({\bar{\alpha}}\otimes(\psi'|_{\tilde{H}} ))=L_{K'}({\bar{\alpha'}}\otimes(\psi'|_{\tilde{H'}}))$, \newline where $\psi =\Ind^{\tilde{G}}_{\tilde{H}} (\alpha)$ and $\psi' =\Ind^{\tilde{G}}_{\tilde{H'}} (\alpha')$;
	\item Induced representations $\psi$ and $\psi'$  are isomorphic.
\end{enumerate}
Moreover, if those conditions hold then $K$ and $K'$ isomorphic as extensions of $\mathbb F_q(t)$.
\end{theorem}

The paper has the following structure: in the next section we give a proof of theorem $[\ref{Lfunct}]$. After that we study arithmetical equivalence for global function fields: we provide few explicit examples of non-isomorphic, but arithmetically equivalent global function fields, discuss an algorithm to construct two-parametric family of such pairs with any finite base field and briefly review properties of such fields.  In the next section we give a proof of theorem [\ref{Group}] and in the last section we give a proof of theorem [\ref{Mn}]. 
\section{On the L-functions criteria}

In this section we are going to prove our main theorem, but before that, let us first consider one particular example. 

\begin{ex} 
Consider two elliptic curves $E$ and $E'$ over $\mathbb F_7$, affine part of which defined by equations $y^2=x^3+1$ and $y^2=x^3+3x+1$ respectively. Let us denote by $K$ and $K'$ the corresponding function fields. One checks that $$\zeta_K(T) = \frac{7T^2 + 4T + 1 }{(1-T)(1-7T) }  = \zeta_{K'}(T),$$ where $T=7^{-s}$. Hence by the theorem of Weil, $E$ and $E'$ are $\mathbb F_{7}$-isogenous, but $j(E) = 0$ and $j(E')=2$ so they are not isomorphic even over the algebraic closure $\overbar{\mathbb F_7}$ and hence $K \not \simeq K'$.  
\end{ex}
   
In the above example we have quadratic extensions $\mathbb F_7(\sqrt{f_i(x)}) / \mathbb F_7(x)$, where $f_1(x) =x^3+1 $ and $f_2(x)=x^3+3x+1$. Obviously those are abelian Galois extensions with Galois group $C_2$. It means that despite the fact that $K$ and $K'$ share the same $\zeta$-function they do not share splitting type(otherwise they must be isomorphic). But we could fix this. Namely, let us consider the common Galois closure $N$. We denote by $G, H, H'$ Galois groups of $\Gal(N / \mathbb F_q(x)), \Gal(N / K), \Gal(N / K' )$, respectevly. We have $G = C_2 \oplus C_2$ and hence there exists a one-dimensional character $\chi$ of $G$ such that $\chi|_H = 1_H$ and $\chi|_{H'}\ne 1_{H'}$. Now $L_K( \chi|_H ) = \zeta_K$ and therefore this function has a pole at $s=1$. But, $L_{K'}(\chi|_{H'})$ is an Artin L-function of a non-trivial abelian character, hence it has no poles, see \cite{Rosen}. Therefore we see that $L_K( \chi|_H ) \ne L_{K'}(\chi|_{H'})$. This idea gives rise to the our main theorem.

\begin{theorem}\label{main}
Let $K$, $K'$ denote two finite separable geometric extensions of $\mathbb F_q(x)$. Len $N$ denote the common Galois closure and $G= \Gal(N / \mathbb F_q(x)) $, $H = \Gal(N / K) $, $H' = \Gal(N / K') $. Let $\alpha$ denotes a complex representation of $H$ and $\alpha'$ denotes a complex representation of $H'$. Let $\rho_1$, \dots $\rho_n$ denote all irreducible complex representations of $G$ and  $\bar{\rho}$ denotes the dual representation of $\rho$. Let $\psi =\Ind^{G}_{H} (\alpha)$ and $\psi' =\Ind^{G}_{H'} (\alpha')$. The following are equivalent: 
\begin{enumerate}
	\item For all $i$ such that $1 \le i \le n$ we have equality of Artin L-functions: $L_{K} ( \alpha \otimes \rho_i|_{H} ) = L_{K'} ( \alpha' \otimes \rho_i|_{H'} )$ 
	\item $L_K({\bar{\alpha}}\otimes(\psi|_{H}) )=L_{K'}({\bar{\alpha'}}\otimes(\psi|_{H'}))$ and 
	$L_K({\bar{\alpha}}\otimes(\psi'|_{H} ))=L_{K'}({\bar{\alpha'}}\otimes(\psi'|_{H'}))$;
	\item Induced representations $\psi$ and $\psi'$ are isomorphic.
\end{enumerate}
\end{theorem}  
\begin{proof}
First we show implication \textbf{from (1) to (3)}. For any fixed representation $\rho$ of $G$ we consider $L_{K} ( \alpha \otimes \rho|_{H} )$. This is a meromorphic L-function with no poles outside $s=0$ and $s=1$, see $\cite{Rosen}$. By properties of Artin L-functions this function has a pole at $s=1$ of order $(\alpha \otimes \rho|_H , 1 )_{H} $, possibly zero. Because of properties of complex representations: $(\alpha \otimes \rho|_H , 1 )_{H}  = (\rho|_H , \bar{\alpha} )_{H} $, where $\bar{\alpha}$ means the dual of the representation $\alpha$. By the Frobenius reciprocity we have $$ (\rho|_{H} , \bar{\alpha} )_{H} = (\rho, \Ind^{G}_{H} (\bar{\alpha}))_{G} .$$ 
In means that equality $L_{K} ( \alpha \otimes \rho_i|_{H} ) = L_{K'} ( \alpha' \otimes \rho_i|_{H'} )$ implies
$$(\rho_i, \Ind^{G}_{H} (\bar{\alpha}))_{G} = (\rho_i, \Ind^{G}_{H'} (\bar{\alpha}'))_{G}.$$ 
Since $\rho_i$ runs over all irreducible representations of $G$ it means that $$\Ind^{G}_{H} (\bar{\alpha}) = \Ind^{G}_{H'} (\bar{\alpha}')$$ and therefore $\Ind^{G}_{H} ({\alpha}) = \Ind^{G}_{H'} ({\alpha'})$.

\textbf{From (3) to (1)}. 
By the Frobenius reciprocity for each $i$, $j \in \{ 1 \dots  n \} $ we have: $$ (\Ind^{G}_{H} ({\alpha} \otimes \rho_i|_{H}), \rho_j  )_{G} = (\alpha \otimes \rho_i|_{H}, \rho_j|_{H})_{H} = (\alpha, (\bar{\rho_i} \otimes \rho_j )|_H )_{H} = ( \Ind^{G}_{H} ({\alpha}),   \bar{\rho_i} \otimes \rho_j  )_{G},$$
By our assumptions $\Ind^{G}_{H} ({\alpha}) = \Ind^{G}_{H'} ({\alpha'})$, therefore we have: $$( \Ind^{G}_{H} ({\alpha}),   \bar{\rho_i} \otimes \rho_j  )_{G} = ( \Ind^{G}_{H'} ({\alpha'}),   \bar{\rho_i} \otimes \rho_j  )_{G},$$ 
and hence for each irreducible representation $\rho_i$,  we have: $$\Ind^{G}_{H} ({\alpha} \otimes \rho_i|_{H}) = \Ind^{G}_{H'} ({\alpha'} \otimes \rho_i|_{H'}).$$

Finally, by the Artin induction property we have that: $$L_K( \alpha \otimes \rho_i |_H) = L_{\mathbb F_q(x)} ( \Ind^{G}_{H} ({\alpha} \otimes \rho_i|_{H}) ),$$ 
and therefore we are done.

\textbf{From (2) to (3)}. As before from equality of L-functions we obtained equality of poles and therefore following equalities: 
$$ (\bar{\alpha}\otimes(\psi|_{H}), 1_H)_H = (\bar{\alpha'}\otimes(\psi|_{H'}), 1_{H'})_{H'}  $$ 
and 
$$ (\bar{\alpha}\otimes(\psi'|_{H}), 1_H)_H = (\bar{\alpha'}\otimes(\psi'|_{H'}), 1_{H'})_{H'}  .$$ 
From the Frobenius reciprocity we have:

$$(\bar{\alpha}\otimes(\psi|_{H}), 1_H)_H = ({\alpha},  { \psi}|_{H} ) _H = (\psi, {\psi})_{G}$$
and
$$ (\bar{\alpha'}\otimes(\psi|_{H'}), 1_{H'})_{H'} = ({\alpha'},  { \psi}|_{H'} ) _{H'} = (\psi', {\psi} )_{G} $$

Therefore assumptions of (2) implies $(\psi, \psi)_G = (\psi, \psi')_G = (\psi', \psi')_G $. Let us consider the scalar of product of the virtual representation $\psi - \psi' $ with itself: $(\psi - \psi', \psi - \psi')_{G} =  (\psi,\psi)_{G} - 2(\psi, \psi')_G + (\psi', \psi')_G = 0 $. Which implies that $\psi$ and  $\psi'$ are isomorphic.

\textbf{From (3) to (2)}

Note that $L_K({\bar{\alpha}}\otimes(\psi|_{H}) ) = L_{\mathbb F_q(x)} (\Ind^G_{H}({\bar{\alpha}}\otimes(\psi|_{H}) ) ).$ Therefore in order to get equality of L-functions it is enough to show:
$$\Ind^{G}_{H}({\bar{\alpha}}\otimes(\psi|_{H}) )  = \Ind^{G}_{H'}({\bar{\alpha'}}\otimes(\psi|_{H'}) ).$$
Let $\rho_i$ runs over irreducible representations of $G$. By the Frobenius reciprocity we have: 

$$ (\Ind^{G}_{H}({\bar{\alpha}}\otimes(\psi|_{H}) ) , \rho_i)_{G} = ( {\bar{\alpha}}\otimes \psi|_{H} , \rho_i|_H )_{H} = ( \bar{\alpha} , \rho_i|_{H} \otimes \bar{\psi}|_{H})_{H} =  (\bar{\psi}, \rho_i\otimes\bar{\psi})_{G}.$$
Since $\psi = \psi'$ we have: 
$$ (\bar{\psi}, \rho_i\otimes \bar{\psi})_{G} = (\bar{\psi'}, \rho_i\otimes \bar{\psi})_{G} = (\bar{\alpha'}, \rho_i|_{H'}\otimes \bar{\psi}|_{H'} )_{H'} = (\bar{\alpha'}\otimes \psi|_{H'} , \rho_i|_{H'})_{H'}= (\Ind^G_{H'}({\bar{\alpha'}}\otimes(\psi|_{H'}) ) , \rho_i)_{G}$$

Which means that two representations are isomorphic: $$\Ind^G_{H}({\bar{\alpha}}\otimes(\psi|_{H}) )= \Ind^G_{H'}({\bar{\alpha'}}\otimes(\psi|_{H'}) ).$$ 
By replacing $\psi$ by $\psi'$ we obtained the second equality of L-functions.
\end{proof}

Note that if $\alpha $ is the trivial representation, then $L_K(\alpha \otimes \rho_H) = L_K(\rho_H)$. Therefore equality of L-functions for each irreducible $\rho$: $L_K(\rho|_H) = L_{K'}(\rho|_{H'})$ implies arithmetical equivalence and vice versa. 

This remark generalizes the fact that equality of zeta-functions in the number field case is the same as arithmetical equivalence. At first sight this generalization to the function field side seems to be not very natural, since it depends on the $k$-rational map of the curve $X$ to  $\mathbb P^{1}$ and not given in the intrinsic terms of $X$, but as we will see in the next section, this map is very important for the notion of arithmetical equivalence: it is possible to map curves $X$ and $Y$ to $\mathbb P^{1}$ in two different ways, such that they function fields are arithmetically equivalent under the first map, but not arithmetically equivalent under the second map. 

\section{On Gassmann Equivalence}
\subsection{Examples}
In order to find examples of arithmetically equivalent function fields we must find a non-trivial example of a Gassmann triple $(G,H,H')$ and solve the inverse Galois problem for $G$. Gassmann triples corresponding to field extensions of degree up to $15$ were classified in $\cite{Bart2}$. It follows that fields with Galois group $G = \Pgl_3(\mathbb F_2)\simeq \Psl_2(\mathbb F_7)$ give rise to at least two non-trivial Gassmann triples: one in degree seven and one in degree fourteen. Also, fields with Galois group $G= \Psl_2(\mathbb F_{11})$ give rise to at least one pair of arithmetically equivalent fields of degree eleven.

Using Magma we compute the Galois group of the splitting field of a given polynomial $f \in \mathbb F_q(x)[y] $ and find all intermediate subfields. By doing that for many different $f$  we find explicit equations of arithmetically equivalent function fields and compare their properties.
\subsubsection{Some Constructions}

 Here are some examples.

\begin{ex}
Let $p=7$, $q=p^2$ and let $\alpha$ be a generator of $\mathbb F^{*}_q$. Consider the function field extension of $\mathbb F_q(x)$ given by $f(y) = y^{p+1} + y - x^{p+1}$. It's splitting field $N$ has degree 168 and Galois group $\Gal(N: \mathbb F_q(x)) \simeq  \Pgl_3(\mathbb F_2)$. Inside this field we have at least two pairs of arithmetically equivalent global function fields: 
\begin{enumerate}
	\item $K_1 : y^7 + 5x^8y^3 + \alpha^4 x^{12} y + 6$ and $K_1' : y^7 + 5x^8y^3 + \alpha^{28}x^{12}y + 6$;
	\item $K_2 : y^{14} + 3x^8y^6 + \alpha^4x^{12}y^2 + 5$ and $K_2' : y^{14} + 3x^8y^6 + \alpha^{28}x^{12}y^2 + 5$;
\end{enumerate} 
\end{ex}

Note that since these fields arise from non-trivial triple $(G,H,H')$ it means that they are not isomorphic as extensions of $\mathbb F_q(x)$, but it may happen that $K$ and $K'$ isomorphic as abstract fields. Indeed, one could check that in this case we have $K_1 \simeq K_1'$ and $K_2\simeq K_2'$ as fields. 

An interesting question is: is it possible to find arithmetically equivalent function fields $K$ and $K'$ that are non isomorphic as abstract fields? It was mentioned in $\cite{GalGr}$ that a result by Serre states that the function field of the normal closure of the field given by $y^{p+1}-xy+1$ over $\mathbb F_p$ has Galois group $\Psl_2(\mathbb F_p)$. By working out this example for $p=7$ and $p=11$ one finds a positive answer to the above question: 
\begin{ex}\label{Ex1}
Consider the curve defined by the affine equation $y^8-xy+1$ over $\mathbb F_7$. The corresponding function field $N$ of the normal closure has degree 168 and the Galois group is $G = \Pgl_3(\mathbb F_2) \simeq \Psl_2(\mathbb F_7)$. Inside this field we have at least two pairs of arithmetically equivalent global function fields: 
\begin{enumerate}
	\item $K_1 : y^7 + 2y^3 + 2y + 6x^2$ and $K_1': y^7 + y^3 + 5y + 4x^2$;
	\item $K_2 : y^{14} + 4y^6 + 5y^2 + 5x^2$ and $K_2' : y^{14} + 4y^6 + 2y^2 + 5x^2$;
\end{enumerate} 
\end{ex}

Being arithmetically equivalent they share the same zeta-function and therefore their Weil-polynomials $f_K(T)$ are the same. Since $f_K(1) = h$ is the class number, we have that in contrast to the number fields they share the same class number, see $\cite{Bart1}$. But they have different class group, hence they are not isomorphic. Indeed according to Magma we have: 
$$\Cl(K_1) \simeq \Cl(K_2) \simeq \mathbb Z/ 8 \mathbb Z $$ but $$\Cl(K_1') \simeq \Cl(K_2') \simeq \mathbb Z/ 4 \mathbb Z \oplus \mathbb Z / 2\mathbb Z.$$  The fact that $\Cl(K_1) \simeq \Cl(K_2) $ and $\Cl(K_1') \simeq \Cl(K_2')$ is not coincidence: $K_1 \simeq K_2$ and $K_1' \simeq K_2'$ as abstract fields. 
 Another important remark here is that the genus of $K_1$ and $K_1'$ is one. They have a rational point over $\mathbb F_7$ and therefore correspond to two elliptic curves $E$ and $E'$ defined over $\mathbb F_7$. By considering a Weierstrass models of $E$ and $E'$ one gets degree two extensions of $\mathbb F_7(x)$, such that they are not arithmetically equivalent. More concretely, the curve defined by $y^7 +2y^3 +2y+6x^2$ is isomorphic to the elliptic curve $E_1$ defined by $y^2-x^3-x$ and the curve given by the equation $y^7 + y^3 + 5y + 4x^2$ is isomorphic to the elliptic curve $E_2$ defined by $y^2-x^3-3x$. This illustrates that the notion of arithmetical equivalence completely depends on the function field embedding.

\begin{ex}
Consider the curve defined by the affine equation $y^{12}-xy+1$ over $\mathbb F_{11}$. The corresponding function field $N$ of the normal closure has degree 660 and the Galois group is $G = \Psl_2(\mathbb F_{11})$. Inside this field we have at least one pair of arithmetically equivalent global function fields: $K_1 : y^{11} + 2y^5 + 8y^2 + 10x^2$ and $K_1': y^{11} + 2y^5 + 3y^2 + 10x^2$.
\end{ex} 
One checks that $K_1$ and $K_1'$ are not isomorphic as global fields, also have genus one and that $\Cl(K_1) \simeq \Cl(K_1') \simeq \mathbb Z/ 12 \mathbb Z$. 

\subsubsection{Construction by Torsion Points on Elliptic Curves}

All the above examples work only for some particular characteristic $p$ of the base field. Moreover, for any example of non-isomorphic Gassmann equivalent pair $(K,K')$ given above fields $K$ and $K'$ actually become isomorphic after the constant field extension. It means that corresponding curves $X$ and $Y$ are twists of each other. In this section we discuss an algorithm to construct examples with arbitrary characteristic $p$ of  the ground field, provided $p$ is grater than three. By using this approach we found geometrically non-isomorphic arithmetically equivalent global fields. 

Let $l$ denote a prime number. As it follows from \cite{Bart2} extensions with Galois group $G \simeq \gl_2(\mathbb F_l)$ play an important role in the construction of arithmetically equivalent fields. If $E$ is an ordinary elliptic curve defined over $\mathbb Q$, then the group $E[l]$ of $l$-torsion points of $E$ allows us to construct arithmetically equivalent number fields, as in \cite{Bart3}. But in contrast to the number field case, in the function field settings torsion points on elliptic curves over $\mathbb F_q(t)$ do not always allow to construct extensions with Galois group isomorphic to $\gl_2(\mathbb F_l)$. The crucial difference appears because of constant field extensions.

More concretely, consider the function field $F$ of the projective line defined over $\mathbb F_q$:  $F \simeq \mathbb F_q(t)$, where $q=p^m$, $p$ is prime. Suppose for simplicity that $p>3$ and pick parameters $a$, $b \in F$. Consider an elliptic curve $E$ over $F$ defined by the equation $y^2 = x^3 + ax +b$. For any prime number $l \ne p$ let us consider $\phi_{l,E}(u)$ the $l$-division polynomial of $E$. This is a polynomial with roots corresponding to $x$-coordinates of $l$-torsion points of the elliptic curve $E$, for example: $$\phi_{3,E}(u) = 3u^4+au^2+12bu-a^2.$$

Finally, let $R(t, y) = \Res_{x}(\phi_{l,E}(x), y^2-(x^3+ax+b))$ be the resultant with respect to $x$. This is a polynomial in $t$ and $y$, roots of which correspond to the coordinates of $l$-torsion points of $E$. Generically this is separable polynomial and it generates the finite field extension $K(y)$ of $\mathbb F_q(t)$: $ K(y) = \frac{\mathbb F_q(t,y)}{R(t,y)} $. We will denote the Galois group of the normal closure of $K$ over $F$ by $G$. Let $H$ be the subgroup of $\mathbb F^{\times}_l$ generated by $q$. The analogue of the so-called \emph{Serre's open image theorem} for function fields proved by Igusa in 1959 states that for big enough $l$ depending on $q$ we have the following exact sequence, see \cite{Igusa2}: 

$$ 1 \to \SL_2(\mathbb F_l) \to G \to H  \to 1 .$$
Moreover, in this sequence $\SL_2(\mathbb F_l)$ corresponds to the geometric extension of $F$ and $H$ corresponds to the constant field extension. If $q = 1 \mod l$ then $H$ is trivial and we obtain a geometric extension with $G \simeq \SL_2(\mathbb F_l)$. By taking a quotient of $G$ by $\pm 1$, we will get $\Psl_2(\mathbb F_l)$. The action of $\pm 1$ is given by gluing points with the same $x$-coordinate. Therefore, the splitting field of $\phi_{l,E}(x)$ is the geometric extension of $\mathbb F_q(t)$ with Galois group $\Psl_2(\mathbb F_l)$. Now if $l=7$ or $l=11$ we obtain a family  of arithmetically equivalent pairs.

\begin{ex}
In the above settings let $p=29$ and $l=7$, $a=t$, $b=t+1$. Then: $\phi_{7,E}(x)$ is a polynomial of degree $24$. The spliting field of $\phi_{7,E}(x)$ is a finite geometric extension $K / \mathbb F_{29}(t)$ with the Galois group isomorphic to $\Psl_2(\mathbb F_7)$. Inside this normal closure following two arithmetically equivalent fields are not isomorphic:
 $$K [x] /( x^7 + 20tx^6 + 14t^2x^5 + (6t^3 + 11t^2 + 22t + 11)x^4 + (5t^4 + 23t^3
    + 17t^2 + 23t)x^3 + $$ $$+(20t^5 + 13t^4 + 26t^3 + 13t^2)x^2 + (5t^6 +
    20t^5 + 5t^3 + 21t^2 + 14t + 18)x +$$ $$ + 23t^7 + 26t^6 + 19t^5 + 10t^4 +
    5t^3 + 13t^2 + 25t)$$

       and 
    $$K[x] / (x^7 + 16tx^6 + 2t^2x^5 + (18t^3 + 10t^2 + 20t + 10)x^4 + (27t^4 + 3t^3
    + 6t^2 + 3t)x^3 + $$ $$ + (27t^5 + 17t^4 + 5t^3 + 17t^2)x^2 + (t^6 + 7t^5 +
    16t^4 + 15t^3 + 12t^2 + 8t + 2)x + $$ $$ + 28t^7 + t^6 + 2t^5 + t^4);$$
\end{ex}

According to Magma function fields given above have genus 1 and a $\mathbb F_{29}$-rational point, therefore they are isomorphic to the function fields of two elliptic curves. Those elliptic curves have different j-invariant, namely 16 and 15 respectively. Therefore, they are geometrically non-isomorphic.

 \subsection{Properties of Arithmetically Equivalent Fields}

In this section we will briefly discuss common properties of arithmetically  equivalent global fields that will shed some light on the previous examples.
We start from one remarkable statement proved in the article $\cite{Perl}$: 
\begin{lem}
Let $G$ be a finite group and $H \subset G$ a subgroup of index $n$. Suppose one of the following conditions holds:
\begin{enumerate}
	\item $n \le 6$;  
	\item $H$ is cyclic;
	\item $G = \mathbb S_n$ the full symmetric group of order $n$;
	\item $n=p$ is prime and $G = \mathbb A_p$ is the alternating group of order $p$.
\end{enumerate}
then any Gassmann triple $(G,H,H')$ is trivial. 
\end{lem}

Taking into account our main theorem this statement has the following application to the function field side: 

\begin{cor}
Let $K$ be a finite separable geometric extension of $\mathbb F_q(t)$ of degree $n$ and let $N$ be its Galois closure with Galois group $G$. Let $H$ be a subgroup of $G$ such that $K = N^{H}$. Suppose one of the conditions from the previous lemma holds. Then for any $H' \subset G$ the fields $K$ and $K'=N^{H'}$ are isomorphic if and only if for each irreducible representation $\rho$ of $G$ we have $L_K(\rho|_H) = L_{K'} (\rho|_{H'})$.
\end{cor}

\subsubsection{Adele Rings}
Let $K$ be a global field and let $A_K$ denote the adele ring of $K$. By definition this is the restricted product of all  local completions $K_v$ with respect to $\mathcal O_{v}$, where $v$ denotes a place of $K$. It has a topology coming from restricted product and therefore it is a topological abelian group.

The first remarkable fact is that in the number field case we have the following implications: $A_K \simeq A_L \Rightarrow \zeta_K = \zeta_L \iff $ $K$ and $L$ arithmetically equivalent.  And moreover there exists an example of arithmetically equivalent number fields with non-isomorphic adele rings, see $\cite{Perl}$. 

On the other hand in the function field side we have the following: 

$A_K \simeq A_L \iff \zeta_K = \zeta_L \Leftarrow$  $K$ and $L$ arithmetically equivalent. For the proof of equivalence see $\cite{Turn}$. Roughly speaking the reason here is that in the function fields case all local completions $\mathcal O_v$ depend only on the degree of $v$ whereas in the number field case they also depend on the primes below it. 

\subsubsection{Ideal Class Group}

Arithmetically equivalent function fields share the same zeta-function and therefore they also share the same class-number. But their class-groups may be different, as in example [\ref{Ex1}]. Nevertheless we have a good bound on that difference. Namely, to each Gassmann-triple $(G,H,H')$ Perlis in \cite{Perl3}  attached a natural number $v$, which divides the order of $H$. Suppose that $K$ and $K'$ are two number fields corresponding to the triple $(G,H,H')$, then: if the prime number $l$ does not divide $v$, then the $l$-part of the class-group of $K$ and $K'$ are isomorphic: $\Cl_l(K) \simeq \Cl_l(K') $. His argument works in the following way: first for any Gassmann triple $(G,H,H')$ let us fix an isomorphism $\alpha$ between $G$-modules $\alpha: \mathbb Q(G/H) \simeq \mathbb Q(G/H')$. Once isomorphism $\alpha$ is fixed one could also fix a basis of the vector spaces $\mathbb Q(G/H)$, $\mathbb Q(G/H')$ and then $\alpha$ could be written as matrix $M_{\alpha}$. Let $v_{\alpha} = \det(M_{\alpha})$, which obviously depends only on the $\alpha$ and not on the basis of the vector spaces. Now given a Gassmann triple $(G,H,H')$ he defined a natural number $v = \gcd_{\alpha}(|v_{\alpha}|),$ where $\alpha$ runs over all isomorphism such that $M_{\alpha}$ has integral coefficients. 

On the other hand, from this isomorphism $\alpha$ he constructed a homomorphism $\phi_{\alpha}$ of multiplicative groups $\phi_{\alpha} : K^{*} \to (K')^*$. Apparently this map factors through fraction ideals and therefore induced morphism between ideal class-group. This map has kernel and co-kernel and R.Perlis proved that order of those groups divides natural number $v_{\alpha}$ associated to $\alpha$.  From this one easily deduced the argument about isomorphism of $l$-parts of class-groups for $l$ not dividing $v$.  This argument also works in the function fields case. It was mentioned in \cite{Perl3} that for the Gassmann triple $(G,H,H')$ with $G \simeq \Psl_2(\mathbb F_7)$ one has $v=8$ and therefore for each pair $(K, K')$ of arithmetically equivalent function fields coming from this triple and each prime number $l \ne 2$: $\Cl_l(K) \simeq \Cl_l(K') $. 

\section{On Monomial Representations}
The main purpose of this section is to prove theorem [\ref{Group}]. Before doing that let us recall some basic facts from the theory of induced representations. Let $G$ be a finite group and $H$ a subgroup. Let $\chi$ be a one-dimensional representation of $H$. Consider the induced representation $\psi$ of $G$: $\psi~=~\Ind^{G}_{H} \chi$. By definition $\psi$ acts on the vector space $V$ which could be associated with the direct sum of lines $\oplus \mathbb C_{g_i}$ where each $\mathbb C_{g_i}$ corresponds to the $i$-th left coset $G/H$. Such a pair $(\psi, \oplus \mathbb C_{g_i})$ is called a \emph{monomial representation}. Let $H'$ be another subgroup of $G$ and $\psi' = \Ind^{G}_{H'} \chi'$ for one-dimensional $\chi'$ of $H'$. We will say that we have morphism of pairs $(\psi, \oplus \mathbb C_{g_i})$,  $(\psi' , \oplus \mathbb C_{g'_j})$ if we have a morphism of representations $f$: $\psi \to \psi'$ such that for each line $\mathbb C_{g_i}$ we have $f(\mathbb C_{g_i}) \subset \mathbb C_{g'_j}$ for some $j$.  
\begin{lem}
Suppose we have an isomorphism of monomial representations $(\psi, \oplus \mathbb C_{g_i} )~=~(\psi' , \oplus \mathbb C_{g'_j})$. Then $H$ is a conjugate of $H'$ in $G$.
\end{lem}

\begin{proof}
For the reference see \cite{Mono}.
\end{proof}

\begin{ex}
Let $G$ be a group of multiplicative quaternions with generators $a$ and $b$. Consider the subgroups $H_a = \{1, a, -1, -a  \}$ and $H_b = \{1, b, -1, -b \}$. Let $\chi_a$ be an isomorphism $H_a \simeq \mu_4^{*}$ sending element $a$ to $i$. Let $\chi_b$ be the same character for $H_b$. Then one has $\Ind^G_{H_a} \chi_a \simeq \Ind^G_{H_b} \chi_b $ as representations, but not as monomial representations.

\end{ex}

Let us recall settings for theorem [\ref{Group}]. 
 Let $G$ be a finite group and $H$ a subgroup of index $n$ and $C_l=\mu_l$ be a cyclic group of order $l$, where $l$ is an odd prime. Let us consider semi-direct products $\tilde{G}=C_l^{n}\rtimes G$ and $\tilde{H}=C_l^{n}\rtimes H$, where $G$ acts on $C_l^{n}$ by permuting its component as cosets $G/H$. Let $g_1, \dots g_n$ be representatives of left cosets $G = \cup_i g_iH$. Without loss of generality we assume that $g_1 = e$ is the identity element. We define $\chi$ to be the homomorphism from $\tilde{H} \to  \mu_l$, sending an element $(c_1, \dots, c_n, g )$ to $ c_1$. This is indeed a homomorphism, since $H$ fixes the first coset. Then the following is true: 
 
 \begin{theorem}[Bart de Smit]
For any  subgroup $\tilde{H}' \subset \tilde{G}$ and any abelian character $\chi'$ : $\Tilde{H}' \to \mathbb C^{*}$ if $\Ind_{\tilde{H}'}^{\tilde{G}} (\chi') =  \Ind_{\tilde{H} }^{\tilde{G}} (\chi ) $ then $\tilde{H'}$ and $\tilde H$ are conjugate in $\tilde{G}$.
  \end{theorem}
\begin{proof}
\begin{bf} Step 1.\end{bf} Let $g_1$, \dots ,$g_n$ be a representatives of cosets $G/H$ with $g_1$ equals to the identity element. Note that $g_i$ for $i \ne 1$ cannot fix the first coset. Consider cosets $\tilde{G} / \tilde{H}$.  We claim that each such coset for $i>1$ could be represented as $\gamma_i = (1,1, \dots, 1 , g_i)$, where $g_i \in G/H$. This is true since elements of the form $(\zeta_1,\zeta_2, \dots , \zeta_n,1)$ are in $\tilde{H}$, where $(\zeta_1,\zeta_2, \dots , \zeta_n) \in C_l^{n}$.

\begin{bf} Step 2.\end{bf}Let us consider element $\alpha =  ( \zeta, 1 ,\dots, 1,\dots, 1) \in \tilde{H}$ where $\zeta \in \mu_l$, $\zeta \ne 1$ is in the first position. Such element fixes each coset $\gamma_i \tilde{H}$. Therefore if $\psi = \Ind_{\tilde{H} }^{\tilde{G}} (\chi )$ then $\psi(\alpha)$ is a diagonal matrix with $l$-th roots of unity on the diagonal. Moreover, it is the matrix with the first element is $\zeta$ on the diagonal and each other diagonal element equals to one. Indeed, by definition of induced representation on the $i$-th position we have $\chi( \gamma_i^{-1} \alpha \gamma_i  )$ and it is easy to see that $\gamma_i^{-1} \alpha \gamma_i$ has first 1 on the first position, provided $i \ne 1$.

\begin{bf} Step 3.\end{bf}We claim that $\psi'(\alpha_i)$ is also a diagonal matrix, where $\psi'~=~\Ind_{\tilde{H}'}^{\tilde{G}} (\chi')$. We know that this  is a matrix with exactly one non-zero element in each row and column. Suppose it is not a diagonal, therefore it changes at least two elements and hence trace of this matrix is $\sum_{k=1}^{n-2} \zeta_i $, where $\zeta_i$ are roots of unity. Since $\psi = \psi'$ we have $n-1 + \zeta = \sum_{k=1}^{n-2} \zeta_i $, which can't be true since the absolute value of the left hand side is strictly bigger than $n-2$. Here we use the fact that $l>2$ and therefore $\zeta \ne \pm 1$.

\begin{bf} Step 4.\end{bf} Let $A$ be an isomorphism of representations $\psi$ and $\psi'$. We will show that it is \emph{an isomorphism of monomial representations} $(\psi, \oplus \mathbb C_i )~=~(\psi' , \oplus \mathbb C_j)$. 
Indeed, it suffices to show that in the given basis $A$ is written as permutation matrix. Suppose it is not and therefore we have at least two non-zero elements in one columm. Also it has another non-zero element in some of those two rows, otherwise $\det(A)$ must be zero which is not since $A$ is an isomorphism. We have $ A \psi(\alpha) = \psi'(\alpha) A $ which is easy to calculate since $\psi(\alpha)$ and $\psi'(\alpha)$ are diagonal. By comparing elements from left and right hand sides one has $\zeta = 1$ which leads to the contradiction.
\end{proof}

\section{The Final Step}
In this section we will prove theorem [\ref{Mn}]. We will denote by $F$ rational function field with the base field $\mathbb F_q$: $F = \mathbb F_q(t)$, where $q=p^m$, $p$ is prime. It is enough to show that for any separable geometric extension $K$ of $F$ of degree $n$, with extension $N$ of $K$, $N$ normal over $F$ and Galois Groups $G=\Gal(N/ F)$ and $H=\Gal(N /K)$ there exist an odd prime $l$ and Galois extension $M$ over $F$ with $\Gal(M / F ) \simeq C_l^{n}\rtimes G$ and $\Gal(M / K) =C_l^{n}\rtimes H$, where $G$ acts on components of $C_l^{n}$ by permuting them as cosets $G/H$.  We will prove this statement in a few steps. 

The Chebotarev density theorem for function fields(see \cite{Rosen} theorem[9.13B]) insures us that for any sufficiently large number $T$ we could find a prime $\mathfrak{p}$ of $F$ which has degree $T$ and splits completely in $N$. Note that if prime splits completely in $N$ then it also splits completely in $K$. Now, we pick an odd prime number $l$  co-prime to the characteristic $p$, to $q-1$, to the order of $G$ and to the class number $h_K$ of $K$. Then we pick a large enough number $T$ divisible by $(l-1)$.
Finally we pick a prime $\mathfrak{p}$ of $F$ of degree $T$ which splits completely in $N$. Let $\mathfrak{b_1}, \dots, \mathfrak{b_n} $ denote primes of $K$ lying above it. We have:
 
\begin{lem}
In the above settings there exists cyclic ramified extension $L_l$ of $K$ of degree $l$ ramifing only at $\mathfrak{b_1}$.
\end{lem}
\begin{proof}
Consider the modulus $\mathfrak{m} \defeq  \mathfrak{b_1}$ and associated ray class group $\Cl_{\mathfrak{m}}(K)$. We will show this group has a subgroup of order $l$. Let $\mathcal{O}_K$ denotes the ring of integers of $K$ with respect to the field extension $K / F$.
Class field theory shows that we have the following exact sequence of abelian groups: 
$$0 \to \mathbb F_q^{*} \to  \left( \mathcal{O}_K/\mathfrak{m} \right )^{*}  \to \Cl_{m} (K) \to \Cl(K)\to 0,$$
We claim that $\Cl_{m}(K)$ contains a subgroup of order $l$ and since $l$ is prime to the order of $\Cl(K)$ the fixed field corresponding to this subgroup is ramified at $\mathfrak{b_1}$.

Indeed the order of $\left( \mathcal{O}_K/\mathfrak{m} \right )^{*}$ is $N(\mathfrak{b}_1) -1 = q^{T} -1 $ , where $N(\mathfrak{a})$ denotes the norm of an ideal $\mathfrak{a}$. Since $T$ is divisible by $(l-1)$ this quantity is divisible by $l$. It follows that the order of $\Cl_{m}(K)$ is divisible by $l$ and therefore we have a cyclic extension of $K$ of degree $l$ which ramifies only at $\mathfrak{b}_1$.
\end{proof}

The next step is to take the common normal closure $M$ of $N$ and $L_l$.
\begin{lem}
The Galois group $\Gal(M/F)$ of the common normal closure $M$ of $N$ and $L_l$ over $F$ is $C_l^{n}\rtimes G$.
\end{lem}
\begin{proof}
By construction $N$ is normal over $F$ and $K=N^{H}$. Consider the set $\Hom(K,N)$ of all embeddings of $K$ into $N$. This has an action of $G$ on it isomorphic to the action of $G$ on $G/H$.  For each element $\sigma_i \in \Hom(K,N)$ consider the field $K^{\sigma_i}$ and corresponding cyclic extension $L^{\sigma_i} = L \otimes_{K^{\sigma_i}} N$. We claim that the composites $NL^{\sigma_i}$ are linearly disjoint over $N$ when $\sigma_i$ runs over the set $\Hom(K, N)$. Indeed, consider the set of primes of $N$ which lie over $\mathfrak{p}$ and ramify in the composite $NL^{\sigma_i}$ over $N$. Since $H^{\sigma_i}$ fixes $K^{\sigma_i}$ this set is invariant under the action of $H^{\sigma_i}$ and not invariant under the action of $g$ for each $g \in G$, $g \not \in H^{\sigma_i}$. Hence all $NL^{\sigma_i}$ ramifies in different primes of $N$ lying above $\mathfrak{p}$. Therefore we have $n$ disjoint $C_l$-extensions $NL^{\sigma_i} / N$ in $M$ and $G$ permutes them as cosets $G/H$. It follows that we have the following exact sequence: 
$$ 1 \to C_{l}^{n} \to \Gal(M/F) \to G \to 1$$
Since the order of $G$ is co-prime to $l$, by the Schur--Zassenhaus theorem(see \cite{Group}) we have a section from $\gamma: G \to \Gal(M/F)$ which means that this sequence split and $\Gal(M/F) \simeq C_l^{n} \rtimes G$ as desired. 

\end{proof}

\newpage

\bibliography{mybib}{}

\begin{thebibliography}{10}

\bibitem{Igusa2}
Andrea Bandini, Ignazio Longhi, and Stefano Vigni.
\newblock Torsion points on elliptic curves over function fields and a theorem
  of {I}gusa.
\newblock {\em Expo. Math.}, 27(3):175--209, 2009.

\bibitem{Bart2}
Wieb Bosma and Bart de~Smit.
\newblock On arithmetically equivalent number fields of small degree.
\newblock In {\em Algorithmic number theory ({S}ydney, 2002)}, volume 2369 of
  {\em Lecture Notes in Comput. Sci.}, pages 67--79. Springer, Berlin, 2002.

\bibitem{GalGr}
John Conway, John McKay, and Allan Trojan.
\newblock Galois groups over function fields of positive characteristic.
\newblock {\em Proc. Amer. Math. Soc.}, 138(4):1205--1212, 2010.

\bibitem{Gunt1}
Gunther Cornelissen, Aristides Kontogeorgis, and Lotte van~der Zalm.
\newblock Arithmetic equivalence for function fields, the {G}oss zeta function
  and a generalisation.
\newblock {\em J. Number Theory}, 130(4):1000--1012, 2010.

\bibitem{Bart3}
Bart de~Smit.
\newblock Generating arithmetically equivalent number fields with elliptic
  curves.
\newblock In {\em Algorithmic number theory ({P}ortland, {OR}, 1998)}, volume
  1423 of {\em Lecture Notes in Comput. Sci.}, pages 392--399. Springer,
  Berlin, 1998.

\bibitem{Bart1}
Bart de~Smit and Robert Perlis.
\newblock Zeta functions do not determine class numbers.
\newblock {\em Bull. Amer. Math. Soc. (N.S.)}, 31(2):213--215, 1994.

\bibitem{Mono}
Bosco Fotsing and Burkhard K{\"u}lshammer.
\newblock Modular species and prime ideals for the ring of monomial
  representations of a finite group.
\newblock {\em Comm. Algebra}, 33(10):3667--3677, 2005.

\bibitem{Arith}
Norbert Klingen.
\newblock {\em Arithmetical similarities}.
\newblock Oxford Mathematical Monographs. The Clarendon Press, Oxford
  University Press, New York, 1998.
\newblock Prime decomposition and finite group theory, Oxford Science
  Publications.

\bibitem{Nagata}
Kiyoshi Nagata.
\newblock Artin's {$L$}-functions and {G}assmann equivalence.
\newblock {\em Tokyo J. Math.}, 9(2):357--364, 1986.

\bibitem{Perl3}
Robert Perlis.
\newblock On the class numbers of arithmetically equivalent fields.
\newblock {\em J. Number Theory}, 10(4):489--509, 1978.

\bibitem{Rosen}
Michael Rosen.
\newblock {\em Number theory in function fields}, volume 210 of {\em Graduate
  Texts in Mathematics}.
\newblock Springer-Verlag, New York, 2002.

\bibitem{Group}
Joseph~J. Rotman.
\newblock {\em An introduction to the theory of groups}, volume 148 of {\em
  Graduate Texts in Mathematics}.
\newblock Springer-Verlag, New York, fourth edition, 1995.

\bibitem{Perl}
R.Perlis.
\newblock On the equation $\zeta_k(s) = \zeta_k'(s)$.
\newblock {\em Journal of Number Theory,Volume 9, Issue 3, Pages 342-360},
  1977.

\bibitem{Perl2}
D.~Stuart and R.~Perlis.
\newblock A new characterization of arithmetic equivalence.
\newblock {\em J. Number Theory}, 53(2):300--308, 1995.

\bibitem{Turn}
Stuart Turner.
\newblock Adele rings of global field of positive characteristic.
\newblock {\em Bol. Soc. Brasil. Mat.}, 9(1):89--95, 1978.

\end{thebibliography}
\bibliographystyle{plain}

\newpage

\tableofcontents

\end{document}